\documentclass{proc-l}

\usepackage{amssymb,upref}
\usepackage{graphicx,psfrag}
\usepackage{color}

\newtheorem{theorem}{Theorem}[section]
\newtheorem{pro}[theorem]{Proposition}
\newtheorem{lemma}[theorem]{Lemma}
\theoremstyle{definition}
\newtheorem{definition}[theorem]{Definition}
\theoremstyle{remark}
\newtheorem{remark}[theorem]{Remark}

\numberwithin{equation}{section}

\newcommand{\Cset}{\mathbb{C}}
\newcommand{\Rset}{\mathbb{R}}
\newcommand{\Tset}{\mathbb{T}}
\newcommand{\Zset}{\mathbb{Z}}

\newcommand{\rmd}{{\rm d}}
\newcommand{\rme}{{\rm e}}
\newcommand{\rmi}{\mskip2mu{\rm i}\mskip1mu}
\DeclareMathOperator{\Order}{O}
\DeclareMathOperator{\am}{am}
\DeclareMathOperator{\cn}{cn}
\DeclareMathOperator{\dn}{dn}
\DeclareMathOperator{\sn}{sn}
\DeclareMathOperator{\sech}{sech}

\newif\iffigures
\figurestrue

\begin{document}

\title{The billiard inside an ellipse deformed by the curvature flow}

\author[J. Damasceno]{Josu\'e Damasceno}
\address{Departamento de Matem\'atica \\
         Universidade Federal de Ouro Preto \\
         35.400--000, Ouro Preto, Brazil}
\email{josue@iceb.ufop.br}

\author[M. J. Dias Carneiro]{Mario J. Dias Carneiro}
\address{Departamento de Matem\'atica \\
         ICEx, Universidade Federal de Minas Gerais \\
         30.123--970, Belo Horizonte, Brazil}
\email{carneiro@mat.ufmg.br}

\author[R. Ram{\'\i}rez-Ros]{Rafael Ram\'{\i}rez-Ros}
\address{Departament de Matem\`{a}tiques \\
         Universitat Polit\`{e}cnica de Catalunya \\
         Diagonal 647, 08028 Barcelona, Spain}
\email{rafael.ramirez@upc.edu}

\thanks{R. R.-R. is supported in part by CUR-DIUE Grant 2014SGR504 (Catalonia)
and MINECO-FEDER Grant MTM2012-31714 (Spain).}



\date{\today}

\commby{???}

\begin{abstract}
The billiard dynamics inside an ellipse is integrable.
It has zero topological entropy, four separatrices in the phase space,
and a continuous family of convex caustics: the confocal ellipses.
We prove that the curvature flow destroys the integrability,
increases the topological entropy,
splits the separatrices in a transverse way,
and breaks all resonant convex caustics.
\end{abstract}

\maketitle

\section{Introduction}

One can shorten a smooth plane curve by moving it in the direction of
its normal vector at a speed given by its curvature.
This evolution generates a flow
(called \emph{curvature flow} or \emph{curve shortening flow})
in the space of smooth plane curves that coincides with
the negative $L^2$-gradient flow of the length of the curve.
That is, the curve is shrinking as fast as it can using only local information. 

M.~Gage and R.~Hamilton~\cite{GageHamilton1986} described the long time
behavior of smooth convex plane curves under the curvature flow.
They proved that convex curves stay convex and shrink to a point as
they become more circular.
This convergence to a ``limit'' circle takes
place in the $C^\infty$-norm after a suitable normalization.
M.~Grayson proved that any embedded planar curve
becomes convex before it shrinks to a point~\cite{Grayson1987}.

The length, the enclosed area, the maximal curvature,
the number of inflection points, and other geometric quantities
never increase along the curvature flow~\cite{ChouZhu2001}.
On the contrary, we present an example of how the curvature flow
can increase the topological entropy of the billiard dynamics inside
convex curves.
The \emph{topological entropy} of a dynamical system is a nonnegative extended
real number that is a measure of the complexity of the system~\cite{KatokH1995}.
To be precise, the topological entropy represents the exponential growth rate
of the number of distinguishable orbits as the system evolves.
Therefore, increasing entropy means a more complex billiard dynamics,
which is a bit surprising since the curvature flow rounds any convex smooth
curve and circles are the curves with the simplest billiard dynamics.

Birkhoff~\cite{Birkhoff1927} introduced the problem of
\emph{convex billiard tables} almost 90 years ago as a way to describe
the motion  of a free particle inside a closed convex smooth curve.
The particle is reflected at the boundary according to the law
``angle of incidence equals angle of reflection''.

If the boundary is an ellipse, then the billiard dynamics is
\emph{integrable}~\cite{ChangFriedberg1988,KozlovTreshchev1991,Tabachnikov1995}.
In particular, billiards inside ellipses have zero topological entropy.
The motion along the major axis of the ellipse corresponds
to a hyperbolic two-periodic orbit whose unstable and stable invariant curves
coincide, forming four \emph{separatrices}.
The points on these separatrices correspond to the billiard
trajectories passing through the foci of the ellipse.
The interior of an ellipse is foliated with a continuous
family of convex caustics: its confocal ellipses.
A \emph{caustic} is a curve inside the billiard table with the property
that a billiard trajectory, once tangent to it, stays tangent after
every reflection.
Caustics with Diophantine rotation numbers persist
under small smooth perturbations of the boundary~\cite{Lazutkin1973},
but \emph{resonant caustics}
---the ones whose tangent trajectories are closed polygons,
   so that their rotation numbers are rational---
are fragile structures that generically
break up~\cite{RamirezRos2006,PintodeCarvalhoRamirezRos2013}.

All these dynamical and geometric manifestations of the integrability
of billiards inside ellipses disappear when the ellipse is slightly
deformed by the curvature flow.

\begin{theorem}\label{thm:MainTheorem}
The curvature flow breaks all resonant convex caustics,
splits the separatrices in a transverse way,
increases the topological entropy,
and destroys the integrability of the billiard inside an ellipse.
\end{theorem}

The proof of this theorem has two steps.
First,
we introduce the subharmonic and homoclinic Melnikov potentials
associated to the perturbation of the ellipse under the curvature flow
following the theory developed
in~\cite{DelshamsRamirez1996,DelshamsRamirez1997,RamirezRos2006,PintodeCarvalhoRamirezRos2013}.
In order to study these Melnikov potentials,
we need several explicit formulas for the unperturbed billiard dynamics
that can be found in~\cite{ChangFriedberg1988,DelshamsRamirez1996}.
Second, we check that none of these Melnikov potentials is constant,
which implies that the separatrices split and
all resonant convex caustics break up.
The loss of integrability follows directly from a theorem of
Cushman~\cite{Cushman1978},
whereas the increase of the topological entropy follows from a
theorem of Burns and Weiss~\cite{BurnsWeiss1995}.

We also find all the critical points of the Melnikov potentials,
so we can locate all primary homoclinic points and all Birkhoff
periodic trajectories, at least for small enough perturbations.
Finally, we relate the homoclinic Melnikov potential to the limit
of the subharmonic Melnikov potential when the resonant caustic tends
to the separatrices.
This is, up to our knowledge, the first time that such relation is
explicitly shown up in a discrete system.
Similar relations in continuous systems (that is, for ODEs) have been
known from the eighties, see~\cite[\S 4.6]{GuckemheimerHolmes1989}.

Our perturbed ellipses are static,
we do not deal with time-dependent billiards.

This paper is strongly inspired by Dan Jane's example~\cite{Jane2007}
of a Riemannian surface for which the Ricci flow increases
the topological entropy of the geodesic flow.
His example is also based in a Melnikov computation,
although the final step of his argument require the numerical
evaluation of some Melnikov function.
On the contrary, our result is purely analytic,
since we characterize in a quite explicit way our Melnikov potentials
using the theory of elliptic functions.

We complete this introduction with a note on the organization of the article.
In Section~\ref{sec:BilliardEllipse} we review some known results
concerning billiards inside ellipses.
The first order deformation of the ellipse under the curvature flow
is given in Section~\ref{sec:EllipseCurvatureFlow}.
We review the Melnikov theory for area preserving twist maps
in the framework of billiards inside perturbed ellipses
in Section~\ref{sec:MelnikovPotentials}.
Finally, we check that these Melnikov potentials are not constant
by analyzing their complex singularities in Section~\ref{sec:Computations}.

\section{The billiard inside an ellipse}\label{sec:BilliardEllipse}

We consider the billiard dynamics inside the unperturbed ellipse
\begin{equation}\label{eq:Ellipse}
Q_0 =
\left\{ (x,y) \in \Rset^2 : x^2/a^2 + y^2/b^2 = 1 \right\},\qquad
0 < b < a. 
\end{equation}
Let $c = \sqrt{a^2-b^2}$ be the semi-focal distance of $Q_0$,
so the foci of $Q_0$ are the points $(\pm c,0)$.
We recall a geometric property of ellipses~\cite{Tabachnikov1995}.
Let
\[
C_\lambda =
\left\{
(x,y) \in \Rset^2 : \frac{x^2}{a^2-\lambda^2} + \frac{y^2}{b^2-\lambda^2} = 1
\right\},\qquad \lambda \not \in \{a, b\},
\]
be the family of \emph{confocal conics} to the ellipse $Q_0$.
It is clear that $C_\lambda$ is an ellipse for $0 < \lambda < b$ and
a hyperbola for $b < \lambda < a$.
No real conic exists for $\lambda > a$.

The fundamental property of the billiard inside $Q_0$ is that any segment
(or its prolongation) of a billiard trajectory is tangent to $C_\lambda$
for some fixed caustic parameter $\lambda > 0$.
The notion of tangency in the degenerate case $\lambda = b$ is the following.
A line is tangent to $C_b$ when it passes alternatively through the foci.

We refer to~\cite{AbramowitzS72,WhittakerW27} for a general background
on Jacobian elliptic functions.
Let us recall some basic facts about them.
Given a quantity $k\in(0,1)$, called the \emph{modulus},
the \emph{complete elliptic integral of the first kind} is 
\[
K = K(k) = \int_{0}^{\pi/2}(1-k^2 \sin^2 \phi)^{-1/2} \rmd\phi.
\]
We also write $K' = K'(k) = K(\sqrt{1-k^2})$.
The \emph{amplitude} function $\varphi = \am t = \am(t,k)$ is defined through
the inversion of the integral
\[
t = \int_{0}^{\varphi}(1-k^{2} \sin^2 \phi)^{-1/2}\rmd \phi.
\]
The \emph{elliptic sine} and the \emph{elliptic cosine}
are defined by the trigonometric relations
\begin{equation}\label{eq:PeriodicVariable}
\sn t = \sn(t,k) = \sin \varphi,\qquad
\cn t = \cn(t,k) = \cos \varphi.
\end{equation}
If the angular variable $\varphi$ changes by $2\pi$,
then the angular variable $t$ changes by $4K$.
Thus, any $2\pi$-periodic function in $\varphi$,
becomes a $4K$-periodic function in $t$.
We will usually denote the functions in $t$
by putting a tilde above the name of the function in $\varphi$.
For instance, the $4K$-periodic parameterization of the ellipse
\begin{equation}\label{eq:SubharmonicParameterization}
\tilde{q}_0: \Rset \to Q_0,\qquad
\tilde{q}_0(t) = (a\sn t, b\cn t),
\end{equation}
is obtained from the $2\pi$-periodic parameterization
\begin{equation}\label{eq:ClassicalParameterization}
q_0:\Rset \to Q_0,\qquad
q_0(\varphi) = (a\sin\varphi, b\cos\varphi).
\end{equation}
Clearly, $\tilde{q}_0(t) = q_0(\varphi)$.
The billiard dynamics associated to the convex caustic $C_\lambda$
becomes a rigid rotation $t \mapsto t + \delta$ in the angular variable $t$.
It suffices to find the modulus $k$ and shift $\delta$
associated to each convex caustic $C_\lambda$.

\begin{lemma}[\cite{ChangFriedberg1988}]
\label{lem:ChangFriedberg}
Once fixed a caustic parameter $\lambda \in (0,b)$,
we set the modulus $k \in (0,1)$ and the shift
$\delta \in (0,2K)$ by the formulas
\begin{equation}\label{eq:ModulusShift}
k^2 = (a^2-b^2)/(a^2 - \lambda^2),\qquad
\sn(\delta/2) = \lambda/b.
\end{equation}
The segment joining $\tilde{q}_0(t)$ and $\tilde{q}_0(t+\delta)$
is tangent to the caustic $C_\lambda$ for all $t \in \Rset$.
\end{lemma}

Let $m$ and $n$ be two relatively prime integers such that $1 \le m < n/2$.
Let $\rho(\lambda)$ be the rotation number of the convex caustic $C_\lambda$.
We want to characterize the convex caustic $C_\lambda$ whose tangent
billiard trajectories form closed polygons with $n$ sides
that makes $m$ turns inside $Q_0$ or, equivalently,
the caustic parameter $\lambda \in (0,b)$ such that $\rho(\lambda) = m/n$.
Such caustic parameter is unique because $\rho:(0,b) \to \Rset$ is
an increasing analytic function such that $\rho(0) = 0$ and $\rho(b) = 1/2$,
see~\cite{CasasRamirez2010}.
Any $(m,n)$-periodic billiard trajectory gives rise to
a $(n-m,n)$-periodic one by inverting the direction of motion.
Hence, a convex caustic is $(m,n)$-resonant if and only if it is
also $(n-m,n)$-resonant.
This explains why we can assume that $m < n/2$.

The caustic $C_\lambda$ is the $(m,n)$-resonant convex caustic if and only if
\begin{equation}\label{eq:ResonantCondition}
n \delta = 4 K m.
\end{equation}
This identity has the following geometric interpretation.
When a billiard trajectory makes one turn around $C_\lambda$,
the old angular variable $\varphi$ changes by $2\pi$,
so the new angular variable $t$ changes by $4K$.
On the other hand, we have seen that the variable $t$ changes
by $\delta$ when a billiard trajectory bounces once.
Hence, a billiard trajectory inscribed in $Q_0$ and circumscribed around
$C_\lambda$ makes exactly $m$ turns around $C_\lambda$ after $n$ bounces
if and only if~(\ref{eq:ResonantCondition}) holds.

From now on, $k$ and $\delta$ will denote the modulus and the shift
defined in~(\ref{eq:ModulusShift}).
We will also assume that relation~(\ref{eq:ResonantCondition}) holds,
since we only deal with resonant caustics.
We will skip the dependence of the Jacobian elliptic functions
on the modulus.

The billiard dynamics through the foci of the ellipse can also be simplified
by using a suitable variable $s \in \Rset$.
If a billiard trajectory passes alternatively through the foci,
its segments tend to the major axis of the ellipse both in future and past.
We consider the change of variables
$(-\pi/2,\pi/2) \ni \varphi \mapsto s \in \Rset$ given by
\begin{equation}\label{eq:NonperiodicVariable}
\tanh s = \sin \varphi,\qquad \sech s = \cos \varphi,
\end{equation}
in order to give explicit formulas for this dynamics.
If $\varphi$ moves from $-\pi/2$ to $\pi/2$,
then $s$ moves from $-\infty$ to $+\infty$.
Thus, any $2\pi$-periodic function in $\varphi$ generates a
non-periodic function in $s$.
We will usually denote the function in $s$ by putting a hat above
the name of the function in $\varphi$.
For instance, the parametrization of the upper semi-ellipse
$Q_0^+ = Q_0 \cap \{ y > 0\}$ given by
\begin{equation}\label{eq:HomoclinicParametrization}
\hat{q}_0: \Rset \to Q_0^+,\qquad
\hat{q}_0(s) = (a \tanh s, b \sech s),
\end{equation}
is obtained from parameterization~(\ref{eq:ClassicalParameterization}).
Clearly, $\hat{q}_0(s) = q_0(\varphi)$.

The billiard dynamics through the foci becomes a constant shift
$s \mapsto s+h$ in the variable $s \in \Rset$ for a
suitable shift $h > 0$.

\begin{lemma}[\cite{DelshamsRamirez1996}]\label{lem:DelshamsRamirez}
Once fixed the semi-lengths $0 < b < a$,
let $c = \sqrt{a^2 - b^2}$ be the semi-focal distance and
let $h > 0$ be the quantity determined by
\begin{equation}\label{eq:CharacteristicExponent}
\sinh(h/2) = c/b,\qquad \cosh(h/2) = a/b, \qquad \tanh(h/2) = c/a.
\end{equation}
The segment from $\hat{q}_0(s)$ to $-\hat{q}_0(s+h)$
passes through the focus $(-c,0)$ for all $s \in \Rset$.
\end{lemma}

Note that $\lim_{s \to \pm \infty} \hat{q}_0(s) = (\pm a,0)$,
which shows up that the trajectories through the foci tend to bounce between
the vertices of the major axis of the ellipse.
It is known that these vertices form a two-periodic hyperbolic
trajectory whose \emph{characteristic exponent} is $h$.
That is, the eigenvalues of the differential of the billiard map
at the two-periodic hyperbolic points are $\lambda = \rme^h$ and
$\lambda^{-1} = \rme^{-h}$.
Following a standard terminology in problems of splitting of separatrices,
we will say that the parameterizations~(\ref{eq:SubharmonicParameterization})
and~(\ref{eq:HomoclinicParametrization}) are \emph{natural
parameterizations} of the billiard dynamics tangent to the
convex caustic $C_\lambda$ and through the foci, respectively.

\begin{remark}\label{rem:One2Four}
We can associate four different billiard trajectories to each
$s \in \Rset$.
The first two ones are $\big( (-1)^n\hat{q}_0(s+nh) \big)_{n \in \Zset}$
and $\big( (-1)^n\hat{q}_0(s-nh) \big)_{n \in \Zset}$,
which have the same starting point $\hat{q}_0(s) \in Q^+_0$
but are traveled in opposite directions.
The last two ones are their symmetric images with respect to the origin:
$\big( (-1)^{n+1}\hat{q}_0(s+nh) \big)_{n \in \Zset}$ and
$\big( (-1)^{n+1}\hat{q}_0(s-nh) \big)_{n \in \Zset}$,
which start in a point on the lower semi-ellipse $Q^-_0$.
Hence, there is a one-to-four correspondence between $s$ and
the homoclinic billiard trajectories inside the ellipse $Q_0$.
Indeed, we should consider $s$ defined modulo $h$,
because $s$ and $s+h$ give rise to the same set of
four homoclinic trajectories.
\end{remark}

The billiard dynamics through the foci corresponds to the
caustic parameter $\lambda = b$, so it should be obtained as
the limit of the billiard dynamics tangent to the convex caustic
$C_\lambda$ when $\lambda \to b^-$.
See Figure~\ref{fig:AlmostFlatCaustic}.
Note that $C_\lambda$ flattens into the segment of the $x$-axis enclosed
by the foci of the ellipse when $\lambda \to b^-$.
We confirm this idea in the following lemma.
We also stress that $\lim_{\lambda \to b^-} \delta \neq h$.
This has to do with the minus sign that appears in Lemma~\ref{lem:DelshamsRamirez}
in front of the point $\hat{q}_0(s+h)$.

\begin{figure}
\iffigures
\psfrag{Q}{$Q_0$}
\psfrag{C}{$C_\lambda$}
\psfrag{qp}{$\tilde{q}_0(t+\delta)$}
\psfrag{qt}{$-\hat{q}_0(s-h)$}
\psfrag{qm}{$\tilde{q}_0(t-\delta)$}
\psfrag{qh}{$-\hat{q}_0(s+h)$}
\psfrag{q}{$\tilde{q}_0(t) = \hat{q}_0(s)$}
\includegraphics[height=3in]{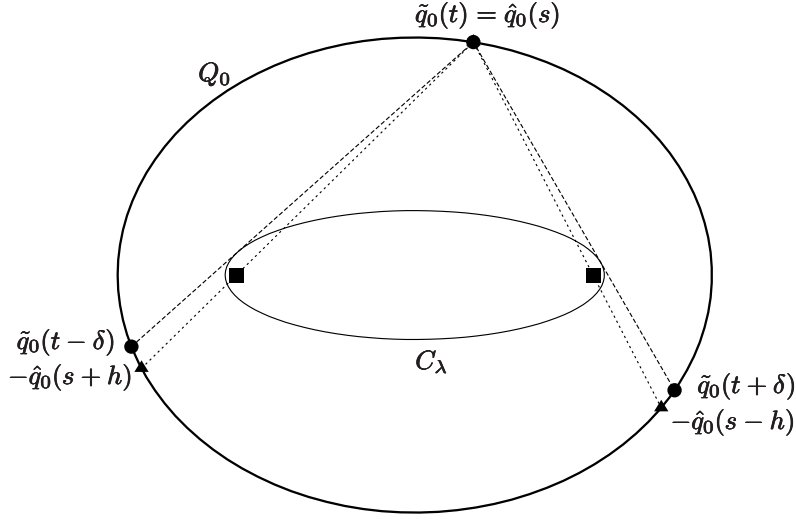}
\else
\vspace{3in}
\fi
\caption{A billiard trajectory (dashed line) tangent to the ellipse
$C_\lambda$ tends to a billiard trajectory (dotted line) through the foci
(the two solid squares) as $\lambda \to b^-$.
The values of $t$ and $s$ are chosen in such a way that
$\tilde{q}_0(t) = \hat{q}_0(s)$.}
\label{fig:AlmostFlatCaustic}
\end{figure}

\begin{lemma}\label{lem:SingularLimits}
Let $k \in (0,1)$, $K=K(k) > 0$, $K'=K'(k) = K(\sqrt{1-k^2})>0$,
and $\delta \in (0,2K)$ be the modulus,
the complete elliptic integral of the first kind,
the complete elliptic integral of the first kind of the complementary modulus,
and the constant shift associated to a convex caustic $C_\lambda$.
Set $\zeta = 2K - \delta \in (0,2K)$.
Then:
\[
\lim_{\lambda \to b^-} k = 1,\qquad
\lim_{\lambda \to b^-} K = +\infty,\qquad
\lim_{\lambda \to b^-} K' = \pi/2,\qquad
\lim_{\lambda \to b^-} \zeta = h,
\]
where $h > 0$ is the characteristic exponent defined
in~\textup{(\ref{eq:CharacteristicExponent})}.
Besides,
\[
\lim_{\lambda \to b^-} \tilde{q}_0(t) = \hat{q}_0(t),\qquad
\lim_{\lambda \to b^-} \tilde{q}_0(t \pm \delta) = -\hat{q}_0(t \mp h),
\]
and both limits are uniform on compacts sets of $\Rset$,
but not on $\Rset$.
\end{lemma}

\begin{proof}
The first limit follows from the definition $k^2 = (a^2 - b^2)/(a^2 - \lambda^2)$.
We know that $\lim_{k \to 1^-} K(k) = +\infty$ and $K'(1) = K(0) = \pi/2$,
which gives the second and third limits.
The fourth limit is a tedious computation using properties of elliptic functions.
The property $\lim_{\lambda \to b^-} \tilde{q}_0(t) = \hat{q}_0(t)$
is a direct consequence of the limits
\[
\lim_{k \to 1^-} \sn(t,k) = \tanh t,\qquad
\lim_{k \to 1^-} \cn(t,k) = \sech t,
\]
which can be found in~\cite{AbramowitzS72}.
Finally,
\[
\lim_{\lambda \to b^-} \tilde{q}_0(t \pm \delta) =
-\lim_{\lambda \to b^-} \tilde{q}_0(t \pm \delta \mp 2K) =
-\hat{q}_0(t \mp h),
\]
where we have used that $\tilde{q}_0(t)$ is $2K$-antiperiodic and
$\lim_{\lambda \to b^-} \zeta = h$.
\end{proof}

\section{An ellipse under the curvature flow}\label{sec:EllipseCurvatureFlow}

Let $\Tset = \Rset/2\pi\Zset$.
Let $Q_0 = q_0(\Tset)$, $q_0:\Tset \to \Rset^2$,
be a closed smooth embedded curve in the plane.
This curve may not be an ellipse.
The $t$-time curvature flow of $Q_0$ is the curve
$Q_t = q_t(\Tset) = q(\Tset;t)$ where the map
$q: \Tset \times [0,\tau) \to \Rset^2$, $q=q(\varphi;t)$, satisfies
the initial value problem
\begin{equation}\label{eq:CurvatureFlow}
\frac{\partial q}{\partial t} = \kappa N,\qquad
q(\cdot,0) = q_0.
\end{equation}
Here, $\kappa$ and $N$ are the curvature
and the unit inward normal vector, respectively.
Observe that $\varphi$ is not, in general, the arc-length parameter.

M.~Gage and R.~Hamilton~\cite{GageHamilton1986} showed that if
$Q_0$ is strictly convex, then the curvature flow is defined for
$t \in [0,\tau)$, where $\tau = A_0/2\pi$ and $A_0$ is the area
enclosed by $Q_0$.
Besides, $Q_t$ shrinks to a point and becomes more circular as $t \to \tau^-$.

Let $Q_0$ be the ellipse~(\ref{eq:Ellipse}).
We want to study a small deformation of $Q_0$ under the curvature flow.
Henceforth, in order to emphasize that we are only interested in
infinitesimal deformations of $Q_0$,
we will denote the infinitesimally deformed ellipse by the symbol $Q_\epsilon$,
instead of $Q_t$.

We consider the elliptic coordinates $(\mu,\varphi)$ associated to
the ellipse $Q_0$.
That is, $(\mu,\varphi)$ are defined by relations
\begin{equation}\label{eq:EllipticCoordinates}
x = c \cosh \mu \sin \varphi,\qquad
y = c \sinh \mu \cos \varphi,
\end{equation}
where $c = \sqrt{a^2-b^2}$ is the semi-focal distance of $Q_0$.
The ellipse $Q_0$ in these elliptic coordinates reads as $\mu \equiv \mu_0$,
where $\cosh \mu_0 = a/c$ and $\sinh \mu_0 = b/c$.
Therefore, the deformation $Q_\epsilon$ of the ellipse $Q_0$
can be written in elliptic coordinates as
\begin{equation}\label{eq:EllipticPerturbation}
\mu = \mu_\epsilon(\varphi) =
\mu_0 + \epsilon \mu_1(\varphi) + \Order(\epsilon^2),
\end{equation}
for some $2\pi$-periodic smooth function $\mu_\epsilon: \Rset \to \Rset$.
If a curve is symmetric with respect to a line,
so is its curvature flow deformation, as long as it exists.
Thus, the deformation $Q_\epsilon$ has the axial symmetries
of the ellipse $Q_0$ with respect to both coordinates axis.
This means that $\mu_\epsilon(\varphi)$ is even and $\pi$-periodic.
Next, we compute the first order term of this function.
That is, we compute the function $\mu_1(\varphi)$.

\begin{lemma}\label{lem:CurvatureFlow}
Let $Q_\epsilon$ be the deformation under the $\epsilon$-time
curvature flow of the ellipse~\textup{(\ref{eq:Ellipse})}.
If we write the deformed ellipse $Q_\epsilon$ as in
equation~\textup{(\ref{eq:EllipticPerturbation})},
then
\begin{equation}\label{eq:mu1}
\mu_1(\varphi) =
\frac{-ab}{(a^2 \cos^2 \varphi + b^2 \sin^2 \varphi)^2}.
\end{equation}
\end{lemma}

\begin{proof}
Let $q:\Tset \times [0,\tau) \to \Rset$,
$q=q_t(\varphi) = q(\varphi;t)$,
be the solution of the initial value problem~(\ref{eq:CurvatureFlow}),
where $q_0(\varphi) = (a\sin\varphi,b\cos\varphi)$.
On the one hand, we obtain from~(\ref{eq:CurvatureFlow}) that
$q_\epsilon(\varphi) =
q_0(\varphi) + \epsilon q_1(\varphi) + \Order(\epsilon^2)$,
where $q_1(\varphi) = \kappa_0(\varphi) N_0(\varphi)$,
\[
\kappa_0(\varphi) =
\frac{ab}{\sqrt{(a^2 \cos^2 \varphi + b^2 \sin^2 \varphi)^3}}
\]
is the curvature of the ellipse $Q_0$ at the point $q_0(\varphi)$, and
\[
N_0(\varphi) =
\frac{-1}{\sqrt{a^2 \cos^2 \varphi + b^2 \sin^2 \varphi}}
(b\sin \varphi, a\cos\varphi)
\]
is the inward unit normal vector of the ellipse $Q_0$ at the point $q_0(\varphi)$.

On the other hand,
we deduce from the elliptic coordinates~(\ref{eq:EllipticCoordinates}) that
\begin{align*}
q_\epsilon(\varphi) & =
(c \cosh \mu_\epsilon(\varphi) \sin\varphi,c\sinh \mu_\epsilon(\varphi) \cos\varphi) \\
& =
(a\sin \varphi,b\cos\varphi) +
\epsilon \mu_1(\varphi) (b \sin \varphi, a \cos \varphi) +
\Order(\epsilon^2).
\end{align*}

By combining these two results, we get that
\[
\frac{-ab}{(a^2 \cos^2 \varphi + b^2 \sin^2 \varphi)^2}
(b \sin \varphi, a \cos\varphi) =
\mu_1(\varphi) (b \sin \varphi, a \cos \varphi),
\]
which implies formula~(\ref{eq:mu1}).
\end{proof}

\section{Subharmonic and homoclinic Melnikov potentials}\label{sec:MelnikovPotentials}

Let us introduce the Melnikov potentials associated to the billiard dynamics
inside a perturbed ellipse that has the form~(\ref{eq:EllipticPerturbation})
in the elliptic coordinates~(\ref{eq:EllipticCoordinates}).
We do not assume now that this perturbed ellipse
is obtained through the curvature flow,
but we still assume that the perturbation preserves the axial symmetries of the
unperturbed ellipse~(\ref{eq:Ellipse}).
This means that
$\mu_\epsilon(\varphi) = \mu_0 + \epsilon \mu_1(\varphi) + \Order(\epsilon^2)$
is an even $\pi$-periodic smooth function.

We define the Melnikov potentials in a way already adapted
to our specific billiard setting and then we list their main properties.
See~\cite{PintodeCarvalhoRamirezRos2013}
(respectively,~\cite{DelshamsRamirez1996,DelshamsRamirez1997})
for a more detailed description of subharmonic (respectively, homoclinic)
Melnikov potentials and their relation with the break up of resonant invariant
curves (respectively, splitting of separatrices) of area-preserving twist maps.

\begin{definition}\label{def:SubharmonicMelnikovPotential}
Let $m$ and $n$ be relatively prime integers such that $1 \le m < n/2$.
The \emph{$(m,n)$-subharmonic Melnikov potential} for the billiard dynamics
inside the perturbed ellipse~(\ref{eq:EllipticPerturbation}) is
\begin{equation}\label{eq:SubharmonicMelnikovPotential}
\tilde{L}^{(m,n)}_1: \Rset \to \Rset,\qquad
\tilde{L}^{(m,n)}_1(t) =
2 \lambda \sum_{j=0}^{n-1} \tilde{\mu}^{(m,n)}_1(t+j \delta),
\end{equation}
where $C_\lambda$ is the $(m,n)$-resonant convex caustic inside $Q_0$,
the modulus $k \in (0,1)$ and the shift $\delta \in (0,2K)$ are defined
in~(\ref{eq:ModulusShift}), and $\tilde{\mu}_1^{(m,n)}(t) = \mu_1(\varphi)$.
Variables $\varphi$ and $t$ are related through the
change~(\ref{eq:PeriodicVariable}).
\end{definition}

\begin{pro}\label{pro:SunharmonicPotential}
The Melnikov potential~\textup{(\ref{eq:SubharmonicMelnikovPotential})}
satisfies the following properties:
\begin{enumerate}
\item
It is an even $\zeta$-periodic smooth function, where $\zeta = 2K - \delta$;
\item
It has critical points at $t = 0$ and $t = \zeta/2$;
\item
If it is not constant, the caustic $C_\lambda$
does not persist under perturbation~\textup{(\ref{eq:EllipticPerturbation})};
\item
If it does not have degenerate critical points and $\epsilon > 0$ is small
enough, then there is a one-to-one correspondence between its critical points
(modulo its $\zeta$-periodicity) and the $(m,n)$-periodic Birkhoff billiard
trajectories inside the deformed ellipse~\textup{(\ref{eq:EllipticPerturbation})}.
\end{enumerate}
\end{pro}

\begin{proof}
The last two claims follow directly from results contained
in~\cite{PintodeCarvalhoRamirezRos2013}.
Next, we prove the first two claims.
On the one hand, $\tilde{L}^{(m,n)}_1(t)$ is $2K$-periodic
because $\tilde{\mu}^{(m,n)}_1(t)$ is so.
On the other hand,
we get from the resonant condition~(\ref{eq:ResonantCondition}) that
$\tilde{\mu}^{(m,n)}(t+n\delta) =
 \tilde{\mu}^{(m,n)}_1(t + 4Km) = \tilde{\mu}^{(m,n)}_1(t)$,
so
\[
\tilde{L}^{(m,n)}_1(t+\delta) =
2 \lambda \sum_{j=0}^{n-1} \tilde{\mu}^{(m,n)}_1(t + \delta + j \delta) =
2 \lambda \sum_{j=0}^{n-1} \tilde{\mu}^{(m,n)}_1(t + j \delta) =
\tilde{L}^{(m,n)}_1(t).
\]
This means that $\tilde{L}^{(m,n)}_1(t)$ is also $\delta$-periodic.
Hence, any linear combination with integer coefficients of the periods
$2K$ and $\delta$ is also a period.
We focus on the integer combination $\zeta = 2K - \delta$ due to the result
presented in Lemma~\ref{lem:SingularLimits}.
Finally, any even $\zeta$-periodic smooth function has critical points
at $t = 0$ and $t = \zeta/2$.
\end{proof}

Since
$\mu_\epsilon(\varphi) = \mu_0 + \epsilon \mu_1(\varphi) + \Order(\epsilon^2)$
is even and $\pi$-periodic, we know that
\[
\breve{\mu}_\infty := \mu_1(-\pi/2) = \mu_1(\pi/2).
\]

\begin{definition}\label{def:HomoclinicMelnikovPotential}
The \emph{homoclinic Melnikov potential} for the billiard dynamics
inside the perturbed ellipse~(\ref{eq:EllipticPerturbation}) is
\begin{equation}\label{eq:HomoclinicMelnikovPotential}
\hat{L}_1: \Rset \to \Rset,\qquad
\hat{L}_1(s) = 2 b \sum_{j \in \Zset} \hat{\mu}_1(s+j h),
\end{equation}
where $\hat{\mu}_1(s) = \mu_1(\varphi) - \breve{\mu}_\infty$.
Variables $\varphi$ and $s$
are related through the change~(\ref{eq:NonperiodicVariable}).
\end{definition}

The series $\sum_{j \in \Zset} \hat{\mu}_1(s + j h)$
converges uniformly on compact subsets of $\Rset$ because
$\lim_{s \to \pm \infty}  \hat{\mu}_1(s) =
 \mu_1(\pm \pi/2) - \breve{\mu}_\infty = 0$ and the variable $\varphi$
tends geometrically fast to $\pm \pi/2$ when $s \to \pm \infty$.
We subtract the constant $\breve{\mu}_\infty$ for this reason.

\begin{pro}\label{pro:HomoclinicPotential}
The Melnikov potential~\textup{(\ref{eq:HomoclinicMelnikovPotential})}
satisfies the following properties:
\begin{enumerate}
\item
It is an even $h$-periodic smooth function;
\item
It has critical points at $s = 0$ and $s = h/2$;
\item
If it is not constant, then the separatrices of the unperturbed
billiard map do not persist under the
perturbation~\textup{(\ref{eq:EllipticPerturbation})};
\item
If it does not have degenerate critical points and $\epsilon > 0$
is small enough, then there is a one-to-four correspondence between its critical
points (modulo its $h$-periodicity) and
the transverse primary homoclinic billiard trajectories inside the deformed
ellipse~\textup{(\ref{eq:EllipticPerturbation})}.
\end{enumerate}
\end{pro}

\begin{proof}
It follows from results contained in~\cite{DelshamsRamirez1996,DelshamsRamirez1997}.
The correspondence is one-to-four because each critical
point gives rise to two different homoclinic ``paths''
(mirrored by the central symmetry with respect to the origin)
and each ``path'' can be traveled in two directions.
See Remark~\ref{rem:One2Four}.
\end{proof}

\section{Computations with elliptic functions}\label{sec:Computations}

Let us assume that the perturbed ellipse~(\ref{eq:EllipticPerturbation})
is the $\epsilon$-time curvature flow of the ellipse~(\ref{eq:Ellipse}),
so that the first order term $\mu_1(\varphi)$ has the form~(\ref{eq:mu1}).
We can not apply the result about non-persistence of resonant
convex caustics established in~\cite{PintodeCarvalhoRamirezRos2013} 
or the result about splitting of separatrices established
in~\cite{DelshamsRamirez1996} to this curvature flow setting,
because the function~(\ref{eq:mu1}) is not entire in the variable $\varphi$.
Nevertheless, many of the ideas developed
in~\cite{DelshamsRamirez1996,PintodeCarvalhoRamirezRos2013}
are still useful.

Let $\tilde{\mu}^{(m,n)}_1: \Rset \to \Rset$ be the function defined by
$\tilde{\mu}^{(m.n)}_1(t) = \mu_1(\varphi)$, so
\begin{equation}\label{eq:tildemu1}
\tilde{\mu}^{(m,n)}_1(t) =
\frac{-ab}{(a^2 \cn^2 t + b^2 \sn^2 t)^2}.
\end{equation}
Here, $C_\lambda$ is the $(m,n)$-resonant convex caustic inside $Q_0$,
the modulus $k \in (0,1)$ and the shift $\delta \in (0,2K)$ are defined
in~(\ref{eq:ModulusShift}), and variables $\varphi$ and $t$ are related
through the change~(\ref{eq:PeriodicVariable}).
We skip the dependence of the Jacobian elliptic functions
on the modulus $k$.

Analogously, let $\hat{\mu}_1:\Rset \to \Rset$ be the function
defined by $\hat{\mu}_1(s) = \mu_1(\varphi) - \breve{\mu}_\infty$, so
\begin{equation}\label{eq:hatmu1}
\hat{\mu}_1(s) = \frac{a}{b^3} - \frac{ab}{(a^2 \sech^2 s + b^2 \tanh^2 s)^2}.
\end{equation}
The key observation in what follows is that~(\ref{eq:tildemu1}) can be
analytically extended to an elliptic function defined over $\Cset$,
whereas~(\ref{eq:hatmu1}) can be analytically extended to a
meromorphic function over $\Cset$.
We list below the main properties of these extensions.

\begin{lemma}\label{lem:tildemu1}
Let $m$ and $n$ be two relatively prime integers such that $1 \le m < n/2$.
Let $C_\lambda$ be the $(m,n)$-resonant elliptical caustic of the ellipse $Q_0$.
Let $\delta \in (0,2K)$ be the shift defined by $\sn(\delta/2) = \lambda/b$,
so relation~\textup{(\ref{eq:ResonantCondition})} holds.
Set $\zeta = 2K-\delta \in (0,2K)$.
The function~\textup{(\ref{eq:tildemu1})} is an even elliptic function of order four,
periods $2K$ and $2K'\rmi$, and double poles in the set
\[
T = T_- \cup T_+, \qquad
T_\pm = t_\pm + 2K \Zset + 2K'\rmi \Zset,\qquad
t_\pm = \pm \zeta/2 + K'\rmi.
\]
It has no other poles.
There exist two Laurent coefficients $\alpha_2,\alpha_1 \in \Cset$,
with $\alpha_2 \neq 0$, such that
\[
\tilde{\mu}^{(m,n)}_1(t_\pm + \tau) =
\frac{\alpha_2}{\tau^2} \pm \frac{\alpha_1}{\tau} + \Order(1),\qquad \tau \to 0.
\]
\end{lemma}

\begin{proof}
We know that the square of the elliptic cosine is an even elliptic function
of order two and periods $2K$ and $2K'\rmi$. 
Thus, the function
\[
f(t) = a^2 \cn^2 t + b^2 \sn^2 t = b^2 + (a^2-b^2) \cn^2 t
\]
has the same properties.
(We have used the identity $\sn^2 + \cn^2 \equiv 1$.)
Hence, the function $f(t)$ has exactly two roots (counted with multiplicity)
in the complex cell
\[
C = \left \{ t \in \Cset : -K \le \Re t < K,\ 0 \le \Im t < 2K' \right\}.
\]
Let us find them.
On the one hand,
the values of the Jacobian elliptic functions at $t = K$ are
\[
\sn K = 1, \qquad \cn K = 0,\qquad
\dn K = \sqrt{1-k^2} = \sqrt{(b^2 - \lambda^2)/(a^2 - \lambda^2)}.
\]
On the other hand, the values of the Jacobian elliptic functions at
$t = \delta/2$ are $\sn(\delta/2) = \lambda/b$,
\[
\cn(\delta/2) = b^{-1}\sqrt{b^2-\lambda^2},\qquad
\dn(\delta/2) = a b^{-1} \sqrt{(b^2 - \lambda^2)/(a^2 - \lambda^2)}.
\]
Therefore, the addition formula for the elliptic sine implies that
\[
\sn (\zeta/2) = \sn(K - \delta/2) =
\frac{\sn K \cn(\delta/2) \dn(\delta/2) - \sn(\delta/2) \cn K \dn K}
     {1-k^2 \sn^2 K \sn^2(\delta/2)} =
\sqrt{a^2 - \lambda^2}/a.
\]
Next, we check that the function $f(t)$ vanishes at the points $t = t_\pm$:
\begin{align*}
f(t_\pm) & = b^2 + (a^2-b^2) \cn^2(\pm \zeta/2 + K'\rmi) \\
& =
b^2 + (a^2-b^2) \left(1 - k^{-2} \sn^{-2}(\pm \zeta/2) \right) = 0.
\end{align*}
We note that $t_\pm = \pm \zeta/2 + K'\rmi \in C$,
so these are the two roots we were looking for and, in addition,
they are simple roots.
From the parity and periodicity of $f(t)$, we also deduce that
\[
f'(t_-) =
f'(-\zeta/2 + K'\rmi) =
-f'(\zeta/2 - K'\rmi) =
-f'(\zeta/2 + K'\rmi) =
-f'(t_+).
\]
Finally, all the properties of the function~(\ref{eq:tildemu1})
follow directly from the fact that $\tilde{\mu}^{(m,n)}_1 = -ab/f^2$.
It suffices to take $\alpha_2 = -ab/(f'(t_+))^2 = -ab/(f'(t_-))^2 \neq 0$.
\end{proof}

\begin{lemma}\label{lem:hatmu1}
Let $h > 0$ be the characteristic exponent~\textup{(\ref{eq:CharacteristicExponent})}.
The function~\textup{(\ref{eq:hatmu1})} is an even meromorphic
$\pi\rmi$-periodic function with double poles in the set
\[
S = S_- \cup S_+, \qquad
S_\pm = s_\pm + \pi\rmi \Zset,\qquad
s_\pm = \pm h/2 + \pi\rmi/2.
\]
It has no other poles.
There exist two Laurent coefficients $\beta_2,\beta_1 \in \Cset$, with $\beta_2 \neq 0$,
such that
\[
\hat{\mu}_1(s_\pm + \sigma) =
\frac{\beta_2}{\sigma^2} \pm \frac{\beta_1}{\sigma} + \Order(1),\qquad
\sigma \to 0.
\]
\end{lemma}

\begin{proof}
The square of the hyperbolic secant is an even meromorphic $\pi\rmi$-periodic function.
Thus, the function
\[
g(s) = a^2 \sech^2 s + b^2 \tanh^2 s = b^2 + c^2 \sech^2 s
\]
has the same properties.
We have used the identities $\sech^2 + \tanh^2 \equiv 1$ and $c^2 = a^2 - b^2$.
Next, we look for all the roots of $g(s)$.
We note that
\[
g(s) = 0 \Leftrightarrow
\cosh^2 s = - c^2/b^2 = -\sinh^2(h/2) = \cosh^2(h/2 + \pi\rmi/2) \Leftrightarrow
s \in S.
\]
We have used that $\cosh^2 s = \cosh^2 r$ if and only if
$s - r \in \pi\rmi\Zset$ or $s + r \in \pi\rmi\Zset$.
These roots are simple.
In fact, if $s_* \in S$,
then $\cosh^2 s_* = -c^2/b^2$ and $\sinh^2 s_* = - a^2/b^2$,
so $g(s_*) = 0$ and $g'(s_*) = -2c^2 \sinh s_* /\cosh^3 s_* \neq 0$.
From the parity and periodicity of $g(s)$, we deduce that
$g'(s_-) = -g'(h/2-\pi\rmi/2) = -g'(s_+)$.
Finally, all the properties of $\hat{\mu}_1(s)$ follow directly
from the fact that $\hat{\mu}_1 = a/b^3 -ab/g^2$.
It suffices to take $\beta_2 = -ab/(g'(s_+))^2 = -ab/(g'(s_-))^2 \neq 0$.
\end{proof}

\begin{pro}\label{pro:SubharmonicNotConstant}
Let $\alpha_2 \neq 0$ be the dominant Laurent coefficient introduced in
Lemma~\ref{lem:tildemu1}.
The $(m,n)$-subharmonic Melnikov potential
\[
\tilde{L}^{(m,n)}_1(t) =
2 \lambda \sum_{j=0}^{n-1} \tilde{\mu}^{(m,n)}_1(t+j \delta),\qquad
\tilde{\mu}^{(m,n)}_1(t) = \frac{-ab}{(a^2 \sn^2 t + b^2 \cn^2 t)^2},
\]
is an even elliptic function of order two with periods $\zeta$ and $2K' \rmi$,
poles in the set
\begin{equation}\label{eq:PolesMelnikovPotential}
\mathcal{T} = t_\star + \zeta \Zset + 2K'\rmi \Zset,\qquad
t_\star = \zeta/2 + K'\rmi,
\end{equation}
and principal parts
\[
\tilde{L}^{(m,n)}_1(t_\star + \tau) =
\left\{
\begin{array}{ll}
4\lambda \alpha_2 \tau^{-2} + \Order(1) \mbox{ as $\tau \to 0$}, &
\mbox{ if $n$ is odd}, \\
8\lambda \alpha_2 \tau^{-2} + \Order(1) \mbox{ as $\tau \to 0$}, &
\mbox{ if $n$ is even}.
\end{array}
\right.
\]
In particular, it is not constant.
Besides, its only real critical points are the points of the set
$\zeta \Zset/2$, and all of them are nondegenerate.
\end{pro}

\begin{proof}
We skip the dependence of $\tilde{\mu}^{(m,n)}_1(t)$ and
$\tilde{L}^{(m,n)}_1(t)$ on $(m,n)$ for simplicity.
The finite sum
$\tilde{L}_1(t) = 2 \lambda \sum_{j=0}^{n-1} \tilde{\mu}_1(t+j \delta)$
can be analytically extended to an elliptic function $\tilde{L}_1: \Cset \to \Cset$
defined over the whole complex plane, see Lemma~\ref{lem:tildemu1}.
The point $t_+ \in \Cset$ is a singularity of $\tilde{\mu}_1(t + j \delta)$
if and only if $t_+ + j\delta \in T = T_- \cup T_+$.
Besides,
\begin{align*}
t_+ + j \delta \in T_+ & \Leftrightarrow
j\delta \in 2K\Zset \Leftrightarrow
2jm \in n \Zset \Leftrightarrow
j \in \{0,n/2\},\\
t_+ + j \delta \in T_- & \Leftrightarrow
(j-1)\delta \in 2K\Zset \Leftrightarrow
2(j-1)m \in n \Zset \Leftrightarrow
j-1 \in \{0,n/2\}.
\end{align*}
We have used that $\delta = 4Km/n$, $t_- = t_+ + \delta - 2K$,
and $\gcd(m,n) =1$.
Equalities $j=n/2$ and $j-1 = n/2$ only can take place when $n$ is even.
Hence, we distinguish two cases:
\begin{itemize}
\item
If $n$ is odd, then $\tilde{\mu}_1(t)$ and $\tilde{\mu}_1(t+\delta)$ are
the only terms in the sum that have a singularity at $t = t_+$,
so that
\begin{align*}
\tilde{L}_1(t_+ + \tau) & =
2\lambda\tilde{\mu}_1(t_+ + \tau) + 2\lambda\tilde{\mu}_1(t_+ + \delta + \tau) +
\Order(1) \\
& =
2\lambda\tilde{\mu}_1(t_+ + \tau) + 2\lambda\tilde{\mu}_1(t_- + \tau) +
\Order(1) \\
& =
4\lambda \alpha_2 \tau^{-2} + \Order(1) \quad \mbox{as $\tau \to 0$}.
\end{align*}
\item
If $n$ is even, then $n \delta /2 = 2Km$ and
$\tilde{L}_1(t) = 4 \lambda \sum_{j=0}^{n/2-1} \tilde{\mu}_1(t+j \delta)$.
We note that $\tilde{\mu}_1(t)$ and $\tilde{\mu}_1(t+\delta)$ are
the only terms in this new sum that have a singularity at $t = t_+$,
so
\begin{align*}
\tilde{L}_1(t_+ + \tau) & =
4\lambda\tilde{\mu}_1(t_+ + \tau) + 4\lambda\tilde{\mu}_1(t_+ + \delta + \tau) +
\Order(1) \\
& =
4\lambda\tilde{\mu}_1(t_+ + \tau) + 4\lambda\tilde{\mu}_1(t_- + \tau) +
\Order(1) \\
& =
8\lambda \alpha_2 \tau^{-2} + \Order(1) \quad \mbox{as $\tau \to 0$}.
\end{align*}
\end{itemize}

Thus, the analytic extension $\tilde{L}_1:\Cset \to \Cset$
has a double pole at $t = t_+$ in both cases, which implies that
$\tilde{L}_1:\Rset \to \Rset$ is not constant.

Next, let us prove that the points in the set $\zeta \Zset/2$ are
the only real critical points of $\tilde{L}_1(t)$,
and all of them are nondegenerate.
The derivative $\tilde{L}'_1(t)$ is odd,
has periods $\zeta$ and $2K'\rmi$,
has triple poles in the set~(\ref{eq:PolesMelnikovPotential}),
and vanishes at the points in the set
$\{0, \zeta/2, K'\rmi \} + \zeta \Zset + 2K'\rmi \Zset$
due to its symmetry and periodicities.
These critical points are nondegenerate and they are the only critical points
because $\tilde{L}'_1(t)$ is an elliptic function of order three.
\end{proof}

\begin{pro}\label{pro:HomoclinicNotConstant}
Let $\beta_2 \neq 0$ be the dominant Laurent coefficient introduced in
Lemma~\ref{lem:hatmu1}.
The homoclinic Melnikov potential
\[
\hat{L}_1(s) =
2 b \sum_{j \in \Zset} \hat{\mu}_1(s+jh),\qquad
\hat{\mu}_1(s) = \frac{a}{b^3} - \frac{ab}{(a^2 \sech^2 s + b^2 \tanh^2 s)^2},
\]
is an even elliptic function of order two with periods $h$ and $\pi\rmi$,
poles in the set
\[
\mathcal{S} = s_\star + h\Zset + \pi\rmi \Zset,\qquad
s_\star = h/2 + \pi\rmi/2,
\]
and principal parts
\[
\hat{L}_1(s_\star + \sigma) = 4 b \beta_2 \sigma^{-2} + \Order(1),
\qquad \sigma \to 0.
\]
In particular, it is not constant.
Besides, its only real critical points are the points of the set
$h \Zset/2$, and all of them are nondegenerate.
\end{pro}

\begin{proof}
The series
$\hat{L}_1(s) = 2 b \sum_{j=0}^{n-1} \hat{\mu}_1(s+jh)$
can be analytically extended to a meromorphic function $\hat{L}_1: \Cset \to \Cset$
defined over the whole complex plane, see Lemma~\ref{lem:hatmu1}.
The point $s_+ \in \Cset$ is a singularity of the $j$-th term
$\hat{\mu}_1(s+jh)$ if and only if
$s_+ + jh \in S = S_- \cup S_+$.
Besides,
\[
s_+ + j h \in S_+ \Leftrightarrow j = 0,\qquad
s_+ + j h \in S_- \Leftrightarrow j = -1,
\]
Here, we have used that $h \in \Rset$ and $s_- = s_+ - h$.
Hence, $\hat{\mu}_1(s-h)$ and $\hat{\mu}_1(s)$ are
the only terms in the sum that have a singularity at $s = s_+$,
so that
\begin{align*}
\hat{L}_1(s_+ + \sigma) & =
2 b \hat{\mu}_1(s_+ - h + \sigma) + 2 b \hat{\mu}_1(s_+ + \sigma) + \Order(1) \\
& =
2 b \hat{\mu}_1(s_- + \sigma) + 2 b \hat{\mu}_1(s_+ + \sigma) +
\Order(1) \\
& =
4 b \beta_2 \sigma^{-2} + \Order(1) \quad \mbox{as $\sigma \to 0$}.
\end{align*}
Therefore, the analytic extension $\hat{L}_1:\Cset \to \Cset$
has a double pole at $s = s_+$, which implies that
the homoclinic Melnikov potential $\hat{L}_1:\Rset \to \Rset$
is not constant.

Finally, the points in the set $h \Zset/2$ are the only
real critical points of $\hat{L}_1(s)$, and all of them are nondegenerate.
This is proved following the same argument at the end of the proof of
the previous proposition.
\end{proof}

The first claims of Theorem~\ref{thm:MainTheorem}
---the break up of all resonant convex caustics and the splitting of
the separatrices in a transverse way---
are a direct consequence of the results above.
K.~Burns and H.~Weiss~\cite{BurnsWeiss1995} proved that such transverse
intersection of separatrices implies that the perturbed system has positive
topological entropy, which gives the third claim of
Theorem~\ref{thm:MainTheorem}.
R. Cushman~\cite{Cushman1978} established that an analytic
area-preserving map with a transverse intersection of stable and
unstable invariant curves can not be integrable, which proves the last claim
of Theorem~\ref{thm:MainTheorem}.
We just note that the billiard dynamics inside an analytic convex curve is
analytic and that the curvature flow preserves the analyticity of the
unperturbed ellipse.

Finally, we establish the relation between the homoclinic Melnikov potential and
the limit of the $(m,n)$-subharmonic Melnikov potential when $m/n \to 1/2$ or,
equivalently, when $\lambda \to b^-$.
We still assume that the perturbed ellipse is obtained by using
the curvature flow, so this is a very specific result.
The relation depends on the parity of the period $n$,
which is a phenomenon that, up to our knowledge,
never takes place in continuous systems.
This is the reason for our interest in it.

\begin{pro}
If $m$ and $n$ are relatively prime integers such that $1 \le m < n/2$,
\[
\lim_{\frac{m}{n} \to \frac{1}{2}} \tilde{L}_1^{(m,n)}(t) =
\mbox{constant} +
\left\{
\begin{array}{rl}
\hat{L}_1(t),   & \mbox{ if $n$ is odd}, \\
2 \hat{L}_1(t), & \mbox{ if $n$ is even},
\end{array}
\right.
\]
uniformly on compact subsets of $\Rset$.
\end{pro}

\begin{proof}
The proof is based on the fact that any elliptic function is determined,
up to an additive constant, by its periods, its poles,
and the principal parts of its poles.
The periods, poles, and principal parts of the subharmonic and homoclinic
Melnikov potentials $\tilde{L}^{(m,n)}_1(t)$ and $\hat{L}_1(s)$
are listed in Propositions~\ref{pro:SubharmonicNotConstant}
and~\ref{pro:HomoclinicNotConstant}, respectively.
We only have to see that the former ones tend to the later ones.

Let $\lambda \in (0,b)$ be the caustic parameter such that $C_\lambda$ is
an $(m,n)$-resonant caustic.
It is known that if $m/n \to 1/2$, then $\lambda \to b^-$.
See~\cite[Proposition 10]{CasasRamirez2010}.
Besides, we have seen in Lemma~\ref{lem:SingularLimits} that
$\lim_{\lambda \to b^-} K' = \pi/2$ and $\lim_{\lambda \to b^-} \zeta = h$.
Thus, it suffices to check that $\lim_{\lambda \to b^-} \alpha_2 = \beta_2$,
where $\alpha_2$ and $\beta_2$ are the Laurent coefficients
introduced in Lemmas~\ref{lem:tildemu1} and~\ref{lem:hatmu1}.
This limit is a straightforward computation.
\end{proof}

\bibliographystyle{amsplain}

\end{document}